\documentclass[12pt,reqno]{amsart}
\usepackage{amsfonts}
\usepackage{graphicx}
\usepackage{mathrsfs}
\usepackage{mathtools}
\usepackage{braket}

\usepackage{amsrefs}
\theoremstyle{plain}
\newtheorem{thm}{Theorem}[section]
\newtheorem{lem}[thm]{Lemma}
\newtheorem{prop}[thm]{Proposition}
\newtheorem{cor}[thm]{Corollary}

\theoremstyle{definition}
\newtheorem{defn}[thm]{Definition}
\newtheorem{example}[thm]{Example}

\theoremstyle{remark}

\numberwithin{equation}{section}

\newcommand{\lset}{\left\{}
\newcommand{\rset}{\right\}}

\newcommand{\reals}{\mathbb{R}}

\newcommand{\inv}{{\mathstrut -1}}

\renewcommand{\phi}{\varphi}

\newcommand{\lam}{\lambda}

\newcommand{\tilU}{\widetilde{U}}
\newcommand{\tilf}{\tilde{f}}

\newcommand{\scrO}{\mathcal{O}}

\DeclareMathOperator{\C}{C}

\newcommand{\funs}[2]{\C^{\infty}(#1;#2)}
\newcommand{\rfuns}[1]{\C^{\infty}(#1)}

\newcommand{\sects}[2]{\Gamma(#1;#2)}
\newcommand{\gsects}[1]{\Gamma(#1)}
\newcommand{\lsects}[1]{\Gamma_{\text{\textup{loc}}}(#1)}

\DeclareMathOperator{\spn}{span}
\DeclareMathOperator{\dom}{dom}
\DeclareMathOperator{\Span}{Span}

\DeclareMathOperator{\Ann}{Ker}

\DeclareMathOperator{\BC}{BC}
\DeclareMathOperator{\fdim}{rk}

\DeclareMathOperator{\Hom}{Hom}
\DeclareMathOperator{\im}{im}

\DeclareMathOperator{\supp}{supp}

\DeclareMathOperator{\maxdim}{maxdim}

\newcommand{\terminology}[1]{\textbf{#1}}

\newcommand{\bspan}{\Span}
\newcommand{\bann}{\Ann}

\newcommand{\annsym}{\perp}

\newcommand{\fm}{\mathcal{F}}
\newcommand{\tilfm}{\widetilde{\fm}}

\newcommand{\tbundle}[2]{\Theta^{#2}(#1)}

\newcommand{\tilG}{\widetilde{G}}

\newcommand{\ratballs}{\mathcal{B}}

\newcommand{\dimset}[1]{\Sigma_{#1}}

\DeclarePairedDelimiter\abs{\lvert}{\rvert}
\let\enorm\abs

\DeclarePairedDelimiter\norm{\lVert}{\rVert}

\newcommand{\bcsp}[2]{\BC^{\infty}(#1;#2)}
\newcommand{\bcspi}[2]{\bcsp{\reals^{#1}}{\reals^{#2}}}
\newcommand{\bcspnm}{\bcspi{n}{m}}

\newcommand{\idl}[1]{\mathcal{I}(#1)}

\begin{document}

\title
{Smooth Distributions are Finitely Generated}

\author{Lance D. Drager}
\address{Department of Mathematics and Statistics\\
Texas Tech University\\
Lubbock, TX \ 79409-1042}
\email{lance.drager@ttu.edu}

\author{Jeffrey M. Lee}
\address{Department of Mathematics and Statistics\\
Texas Tech University\\
Lubbock, TX \ 79409-1042}
\email{jeffrey.lee@ttu.edu}

\author{Efton Park}
\address{Department of Mathematics\\
Texas Christian University\\
Box 298900\\
Fort Worth, TX 76129}
\email{e.park@tcu.edu}

\author{Ken Richardson}
\address{Department of Mathematics\\
Texas Christian University\\
Box 298900\\
Fort Worth, TX 76129}
\email{k.richardson@tcu.edu}

\begin{abstract}
A subbundle
of variable dimension inside the tangent bundle of a
smooth manifold
is called a smooth distribution if it is 
the pointwise span of a family of smooth vector fields.
We prove that all such distributions are finitely generated,
meaning that the family may be taken to be a finite collection.
Further, we show that the space of
smooth sections of such distributions need not be finitely 
generated as a module over the smooth functions.
Our results are valid in greater generality, where the tangent
bundle may be replaced by an arbitrary vector bundle.
The current version of the paper incorporates the corrigendum of 9-21-2026 into the original.
\end{abstract}

\maketitle

\section{Introduction}\label{s:intro}

Let $M$ be a smooth manifold, and let $L$ be a distribution on $M$, 
i.e. a subbundle of the tangent bundle $TM$.
The well-known Frobenius theorem
states that $L$ determines a foliation of $M$ if and only if
$L$ is involutive.  
Recall that the hypotheses of the Frobenius theorem
require that the dimension of the subspace $L_{p}$ is a
constant function of $p\in M$.

In many fields, for example control theory and Poisson geometry, 
 one encounters \terminology{generalized distributions}, where
the subspace $L_{p}$ can have different dimensions at different
points. We call a distribution of constant rank
a \terminology{regular} distribution.
Sussmann \cite{MR0321133} and Stefan \cite{MR577729}
extended the Frobenius theorem to generalized distributions
(see Michor \cite{MR2428390}*{Chapter~1, Section~3} for a nice exposition).
Sussmann and Stefan considered a distribution $L$ to be smooth if 
for each $p\in M$, there are locally defined vector fields that are 
sections of $L$
whose values at $p$ span $L_p$. 

On the other hand, in the theory of exterior differential systems one encounters
generalized
distributions that are defined as the kernels of families of
one-forms.  We call such a distribution \terminology{cosmooth}.  
The subspaces of the cotangent spaces spanned by these one-forms
determine a generalized subbundle of the cotangent bundle.
This motivates studying generalized subbundles of arbitrary vector
bundles.

If $E$ is a vector bundle over $M$, a generalized subbundle $G$
of $E$ is an assignment $p\mapsto G_{p}$ of a subspace $G_{p}$ of the fiber
$E_{p}$ of $E$ over $p$, for each point $p\in M$.  The
interesting cases are where $G$ is smooth or cosmooth.  A section $s$
of $E$ is said to be a \textbf{section of} $G$ if
$s(p)\in G_{p}$ for all $p\in M$.

The book \cite{MR2099139} by Bullo and Lewis has an interesting discussion of generalized
distributions and generalized subbundles; their book was an
inspiration for this paper.

If $F$ is a smooth $k$\nobreakdash{-}dimensional subbundle of $E$ in
the usual sense, then every point has a neighborhood $U$ on which there
are $k$ smooth sections $s_{1}, s_{2},\dots,s_{k}$ whose values form a basis
for $F_{p}$ at every point $p\in U$.  Every smooth section of $F$ over
$U$ can be written as a combination of the sections
$s_{1},\dots,s_{k}$ with smooth coefficients.  To put it in other words,
the set of smooth sections of $F$ over $U$ is a module over the ring of
smooth functions on $U$; this module is finitely generated with
generators $s_{1},\dots,s_{k}$.

In the case of a generalized subbundle $G$, we can find sections of
$G$ on an open
set $U$ that form a basis at each point only if the dimension of $G$ is
constant on $U$.  We can generalize what happens in the regular case
in two directions.

First, following the terminology in Bullo and Lewis \cite{MR2099139}, we say that a
generalized subbundle $G$ of $E$ is \terminology{finitely generated} over 
an open set $U$ if there are smooth sections $s_{1}, s_{2},\dots
s_{k}$ of $G$ over $U$ so that the values $s_{1}(p), s_{2}(p),\dots,
s_{k}(p)$ \emph{span} $G_{p}$ for each $p\in U$.  We say that
$s_{1},\dots,s_{k}$ are generators for $G$ over $U$.  One can
ask if such generators always exist, either locally or globally.

Second, we can consider the set of sections of $G$ over $U$ as a
module over the ring of smooth functions on $U$ and ask if there are
sections $s_{1},\dots, s_{k}$ as above that form a finite set of generators
for this module.  If the $s_{1},\dots, s_{k}$ generate the module,
Bullo and Lewis \cite{MR2099139} call them \terminology{nondegenerate} generators for $G$ over
$U$.  One can ask if the module of sections is finitely generated,
either globally, i.e, when $U=M$, or locally, i.e., for
some neighborhood $U$ of each point.

Bullo and Lewis \cite{MR2099139}
have a discussion that shows that every point
of a real analytic
generalized subbundle has a neighborhood $U$ on which
there are nondegenerate generators.  This follows from the fact that
the ring of convergent power series is Noetherian.  

In this paper, we study these questions in the smooth case.
For the first question, we show that every generalized subbundle of
a vector bundle is \terminology{globally finitely generated}, that is,
there are finitely many globally defined sections whose values span
the generalized subbundle at each point.   Other researchers
have conjectured that this result is not true, even locally (see, for example, Bullo and Lewis
\cite{MR2099139}*{p.~125}).

We obtain a negative answer to the second question.  We give an example
which shows that the module of sections of a generalized subbundle
(even a tangent distribution) need not be finitely generated, even
locally.

\section{Precise Formulations}

Let $M$ be a manifold.  We will assume that all of our manifolds 
and maps are smooth.

If $V$ is a vector space, we denote the space
of smooth functions $M\to V$ by $\funs{M}{V}$.  In case $V=\reals$,
we write this space as $\rfuns{M}$.


The space of smooth sections of a vector bundle $E$
over an open set $U$ is denoted $\sects{U}{E}$, and the space of smooth
globally defined sections of $E$ is denoted $\gsects{E}$.  A \terminology{local
  section} of $E$ is a smooth section of $E$ defined on some open set.
We denote by $\lsects{E}$ the set of local sections of $E$. Thus
$\lsects{E}$ is the union of the spaces $\sects{U}{E}$ as $U$ ranges
over all open sets.  If $s\in \lsects{E}$, we denote the domain
of $s$ by $\dom(s)$.

A \terminology{generalized
subbundle} $G$ of $E$ is an assignment of a
subspace $G_{p}\subseteq E_{p}$ for each point $p\in M$.  
We do not assume the subspaces $G_p$ vary continuously with $p$ 
or have constant dimension.

If $U\subseteq M$, we say a local
section $s$ of $E$ over $U$ \terminology{belongs to $G$} or \terminology{is a
  section of $G$} if $s(p)\in G_{p}$ for all $p\in U$.  The set of
sections of $E$ over $U$ that belong to $G$ is denoted by
$\sects{U}{G}$.  The set of local sections of $E$ that belong to $G$
on their domains is denoted $\lsects{G}$.

Given a generalized subbundle $G$, there need not be any nonzero
smooth sections of $G$. The following condition on $G$ ensures 
a supply of sections of $G$.

\begin{defn}
  A generalized subbundle $G$ of a vector bundle $E$ is \terminology{smooth}
  if for every point $p$ we can find a family of sections
  $s_{1},\dots,s_{k} \in \lsects{G}$ which contain $p$ in the
  intersection of their domains such that
  \begin{equation*}
    G_{p}= \spn\lset s_{1}(p),\dots,s_{k}(p)\rset.
  \end{equation*}

\end{defn}

The next proposition gives an equivalent definition of
smoothness; the elementary proof is omitted.

\begin{prop}\label{thm:1}
  A generalized subbundle of $G$ of a vector bundle $E$ is smooth if
  and only if for every point $p$ and every $v\in G_{p}$ there is some
section $s\in \lsects{G}$ such that $s(p)=v$.
\end{prop}

\begin{defn}
If $\fm \subseteq \lsects{E}$, 
we define a generalized subbundle  $\bspan(\fm)$ of $E$ by
\begin{equation*}
  \bspan(\fm)_{p} = \spn \lset s(p) : s\in \fm,\  p \in \dom(s)\rset.
\end{equation*}
\end{defn}

We follow the convention that if $V$ is a vector space, the
span of $\emptyset \subseteq V$ is $\lset 0 \rset\subseteq V$.
Thus, if $p$ is not in the domain of any element of $\fm$, then
$\bspan(\fm)_{p}=\lset 0\rset$.

For any family $\fm\subseteq \lsects{E}$, the subbundle $\bspan(\fm)$ is smooth by
definition.  Note that a generalized subbundle $G$ is smooth if and only
if
$G = \bspan(\lsects{G})$.

The following observation will be important. Let $G$ be smooth.
For any point $p$, we can find sections $s_{1}\dots s_{k}$ of $G$ defined on
some open neighborhood $U$ of $p$ such that $s_{1}(p),\dots,s_{k}(p)$
form a basis of $G_{p}$.  These sections
will remain linearly independent on some open neighborhood $V\subseteq
U$ of $p$.   Because these sections belong to $G$, at a point $q\in V$
other than $p$, the set $\left\{ s_{1}(q),\dots, s_{k}(q)\right\}$ is a linearly independent
set in $G_{q}$.  Thus
$\dim(G_{q})\geq k=\dim(G_{p})$.   
This implies that $\dim(p)=\dim(G_{p})$ is 
a lower semicontinuous function of $p\in M$.

We say that $p\in M$ is a \terminology{regular point} of $G$ if $\dim$
is constant on a neighborhood of $p$.  The set of regular points is
open and dense, as is well known.

\begin{defn}
  If $\fm$ is a family of local sections of the dual bundle
  $E^{\ast}$, we define a generalized subbundle $\bann(\fm)$ of $E$ by
\[
\bann(\fm)_{p}=\{v\in E_{p}:s(p)(v)=0 \text{ for all }s\in \fm\text{ with }p\in\dom(s)\}.
  \]
\end{defn}
Note that we can always add the globally defined 
zero section of $E^{\ast}$ to $\fm$ without changing $\bann(\fm)$. 

\begin{defn}
A generalized subbundle $G$ of $E$ is \textbf{cosmooth} if
$G=\bann(\fm)$ for some family $\fm\subseteq\lsects{E^{\ast}}$.
\end{defn}

\begin{defn}
  If $F$ is a generalized subbundle of $E^{\ast}$, we define a
  generalized subbundle $F^{\annsym}$ of $E$
 by
  \begin{equation*}
    F^{\annsym}_{p} = \lset v\in E_{p}: \lam(v)= 0\text{ for all
      $\lam\in F_{p}$}\rset.
  \end{equation*}

\end{defn}
Observe that
$
\bann (\fm)=(\bspan(\fm))^{\annsym}
$.

\begin{thm}
  A generalized subbundle $G$ of $E$ is cosmooth if and only if
$G=F^{\annsym}$ for some smooth generalized subbundle $F$ of $E^{\ast}$.
\end{thm}

\begin{proof}
 Suppose $G$ is cosmooth, so that $G=\bann(\fm)=(\bspan(\fm))^{\annsym}$ for
 some family $\fm$ of local sections of $E^{\ast}$.  
 Since $\bspan(\fm)$ is by definition a smooth
generalized subbundle of $E^{\ast}$, the conclusion follows.

Conversely,  suppose that 
$G=F^{\annsym}$ for some smooth generalized subbundle $F$ of $E^{\ast}$.
Let $\fm = \lsects{F}$.  We claim that $F^{\annsym}=\bann(\fm)$.

Let $v\in F^{\annsym}_{p}$.  If $s\in\fm$ and $p\in\dom(\fm)$, then
$s(p)\in F_{p}$ and $s(p)(v)=0$.
Thus, $v\in \bann(\fm)$, and we conclude $F^{\annsym}\subseteq
\bann(\fm)$.

Next, take $v\in \bann(\fm)_{p}$.  Let $\lam\in F_{p}$.
Since $F$ is smooth, Proposition~\ref{thm:1} shows there is an
$s\in\lsects{F}$ with $s(p)=\lam$.  Then $\lam(v)=s(p)(v)=0$, and
thus $v\in F^{\annsym}_{p}$. Therefore $\bann(\fm)\subseteq F^{\annsym}$.
\end{proof}

If $F$ is a smooth generalized subbundle of $E^{\ast}$,
then the function $p\mapsto \dim(F_{p})$ is lower semicontinuous.
The dimension of $F^{\annsym}_{p}$ is $\dim(E_{p})-\dim(F_{p})$.
Hence the function $p\mapsto \dim(F^{\annsym}_{p})$ is \emph{upper}
semicontinuous. Thus a cosmooth generalized
subbundle need not be smooth.

\begin{example}
Let $E$ and $F$ be vector bundles over $M$ and let
$\phi\colon E\to F$ be a vector bundle map.  It is well known
that if $\phi$ has constant rank, the image and kernel of $\phi$ are
regular subbundles.  If $\phi$ does not have constant rank,
the image of $\phi$ is a smooth generalized subbundle of $F$
and the kernel of $\phi$ is a cosmooth generalized subbundle
of $E$.
\end{example}

\begin{defn}
  Let $E$ be a vector bundle over a manifold $M$ and let $G$ be
  a generalized subbundle of $E$.  If $U\subseteq M$ is an open
set, we say \textbf{$G$ is finitely generated over $U$}
(or, \textbf{is finitely spanned over $U$}) if there
are a finite number of sections $s_{1},\dots,s_{k}\in \sects{U}{G}$
so that, for all $p\in U$,
\begin{equation*}
  G_{p} = \spn\lset s_{1}(p),\dots, s_{k}(p)\rset .
\end{equation*}
We say that $s_{1},\dots,s_{k}$ \textbf{generate $G$ over $U$}.
If we can take $U=M$, we say that $G$ is \textbf{globally finitely
  generated}.
\end{defn}

\section{Fr\'echet spaces of smooth functions}

We now introduce the function spaces we need.  We denote
by $\enorm{\cdot}$ the Euclidean norm on any of the spaces
$\reals^{n}$.

If $f\colon \reals^{n}\to \reals^{m}$ is a bounded function, let
\begin{equation*}
  \norm{f} = \sup\lset \enorm{f(x)}: x\in \reals^{n}\rset
\end{equation*}
be the supremum norm.

We say that $f\colon \reals^{n}\to \reals^{m}$ is in
$\bcsp{\reals^{n}}{\reals^{m}}$ if $f$ is $\C^{\infty}$ and $f$ and
all of its partial derivatives are bounded on $\reals^{n}$.  If $f$ is
such a function, we can define seminorms $p_{k}$ for $k=0,1,2,\dotsc$ by
\begin{equation*}
  p_{k}(f) = \max \Set{ \norm[\bigg]{\frac{\partial^{\abs{\alpha}}f}{\partial
    x^{\alpha}}} : \abs{\alpha} = k},
\end{equation*}
where the maximum is taken over all multi-indices $\alpha$ of order
$k$.  We combine the $p_{k}$'s to form seminorms $\norm{\cdot}_{k}$ by
\begin{equation*}
  \norm{f}_{k} = \sum_{j=0}^{k} p_{j}(f).
\end{equation*}
Thus $\norm{f}_{k}\leq
\norm{f}_{k+1}$.  Note that these seminorms are actually norms.

These norms $\norm{\cdot}_{k}$ (or, equivalently, the seminorms
$p_{k}$) induce a topology on $\bcsp{\reals^{n}}{\reals^{m}}$  which
makes $\bcsp{\reals^{n}}{\reals^{m}}$ into a Fr\'echet space.

Recall that a sequence $\lset f_{i} \rset_{i=1}^{\infty}$ in
$\bcsp{\reals^{n}}{\reals^{m}}$ is Cauchy if it is Cauchy with 
respect to the norm $\norm{\cdot}_k$ for each $k$.  
Therefore, to show that
a series $\sum_{i=0}^{\infty}f_{i}$ converges in $\bcspnm$, it
suffices to show that each of the series
\begin{equation*}
  \sum_{i=1}^{\infty} \norm{f_{i}}_{k},\qquad k=0,1,2,\dots,
\end{equation*}
converges. 
If $\sum_{i=1}^{\infty} f_{i}$ converges in $\bcspnm$ then
$f:=\sum_{i=1}^{\infty}f_{i}$ can be evaluated pointwise,
because convergence in $\bcspnm$ implies pointwise convergence.

\section{The Main Theorem}

\begin{thm}\label{thm:2}
 Let $M$ be a connected manifold and let $E$ be a vector bundle over $M$. 
Let $G$ be a smooth generalized subbundle of $E$.  Then $G$ is
globally finitely generated.
\end{thm}

As we will see, it is possible to give an explicit upper bound on the
number of global sections needed to generate $G$ under appropriate
assumptions. The remainder of this section contains the proof of this theorem.
Before we begin the proof, we note this important corollary.

\begin{cor}\label{thm:3}
 Let $M$ and $E$ be as in Theorem~\ref{thm:2}.  If $G$ is a cosmooth
generalized subbundle of $E$, then there are finitely many
globally defined sections $s_{1},\dots,s_{k}$ of $E^{\ast}$ such that
for each $p\in M$,
\begin{equation*}
  G_{p} = \lset v\in E_{p}: s_{1}(p)(v)=0,\dots, s_{k}(p)(v)=0\rset.
\end{equation*}
In other words, $G$ is defined as the kernel of a finite collection of
global sections of $E^{\ast}$.
\end{cor}

We now prove Theorem~\ref{thm:2}.  Let $M$ be a connected manifold and let $E$ be a
vector bundle over $M$.  The fiber dimension of $E$ will be denoted
by $\fdim(E)$.  Let $G$ be a smooth generalized subbundle of $E$.

We begin by
making two reductions in the problem.
First, we invoke the theorem that every vector bundle over a connected
manifold is isomorphic to a subbundle of a trivial bundle. Thus, for
our problem, we may assume that $E$ is a subbundle of a trivial bundle
$\tbundle{M}{m}=M \times \reals^{m}$ for some integer $m$.  This
theorem, without an estimate on $m$, is well known in the case where
$M$ is compact.  The proof in the noncompact case, which gives an
estimate of $m$, uses topological dimension theory.  A reference for
this material is Greub, Halperin and Vanstone \cite{MR0336650}*{p.~77}, 
but we will need to use a more refined treatment of the dimension
theory, as in Munkres \cite{MR0198479} or Engelking \cite{MR1363947}.
The main point for our purposes is that an upper bound for $m$ is
$\fdim(E)(\dim(M)+1)$.

Our generalized subbundle $G\subseteq E$ is now contained in
$\tbundle{M}{m}$ and is smooth when considered as a generalized
subbundle of $\tbundle{M}{m}$.  Thus, to prove the theorem, it will
suffice to consider a smooth generalized subbundle $G$ of a trivial
bundle $\tbundle{M}{m}$.   We identify each subspace $G_{p}$ with a 
subspace of $\reals^{m}$
and identify sections of $\tbundle{M}{m}$, and hence sections of
$G$, with $\reals^{m}$-valued functions.
We will switch between these points of view as convenient.


Next, we can properly embed $M$ in $\reals^{n}$ for some $n$.  This is
not really necessary for our proof to work, but it makes dealing with
the functions spaces involved simpler.  
If $f\colon U \to \reals^{m}$ is a smooth function defined
on an open subset $U$ of $M$, we can find a smooth extension of $f$ to
a function $\tilf\colon \tilU \to \reals^{m}$, where $\tilU$
is an open subset of $\reals^{n}$ such that $\tilU \cap M = U$.
This can be done by a partition of unity argument, but perhaps the
fastest proof is to note that the tubular neighborhood theorem
says there is an open set $\scrO$ in $\reals^{n}$ containing $M$ and
a smooth retraction $r\colon \scrO\to M$ (i.e., $r(p)=p$ for $p\in
M$).  We define $\tilU = r^{\inv}(U)$ and $\tilf=f\circ
r$.

Because $G$ is smooth, it is
the span of the family $\fm = \lsects{G}$ of local sections of
$\tbundle{M}{m}$.   Considering the elements of $\fm$ to be locally
defined vector-valued functions on $M$, we can extend them to
locally defined vector-valued functions on $\reals^{n}$.  This gives
us a family $\tilfm =\{ \tilf : f\in \fm\}$ of locally defined vector-valued functions whose
restriction to $M$ is $\fm$.  

Considering $\tilfm$ as a subset of $\lsects{\tbundle{\reals^{n}}{m}}$,
we define a generalized subbundle $\tilG$ of
$\tbundle{\reals^{n}}{m}$ by $\tilG = \Span(\tilfm)$. For each
$p\in M$, $\tilG_{p}=G_{p}$, and thus the restriction of $\tilG$ to $M$ is
$G$.

Given a set of global generators for $\tilG$, the restriction of these generators
to $M$ determines a set of global generators for
$G$.  Thus, to prove Theorem~\ref{thm:2}, it 
suffices to prove the
following proposition.

\begin{prop}\label{thm:4}
  If $G$ is a smooth generalized subbundle of the trivial bundle
  $\tbundle{\reals^{n}}{m}$, then $G$ is globally finitely generated.
\end{prop}

 The proof of this proposition will occupy most of the
rest of this section.  To begin, we adopt some notation and terminology.

If $B$ is a Euclidean ball in $\mathbb{R}^n$, we denote by $2B$ the ball with the same center and
twice the radius.  Let $\ratballs$ denote
the set of all balls of rational radius centered at points that
have rational coordinates; $\ratballs$ is a countable basis
for the topology of $\reals^{n}$.

For $0\leq d \leq m$, let 
\begin{equation*}
\dimset{d} = \lset p\in \reals^{n}: \dim(G_{p})=d\rset.
\end{equation*}

Fix $d\geq 1$ such that $\dimset{d}\ne\emptyset$.  Our goal now
 is to construct finitely
many globally defined sections which span $G$ at each point of $\dimset{d}$. Note that spanning 
is automatic for points in
$\dimset{0}$.

The usual Euclidean metric on $\reals^{m}$ induces a metric on
the bundle $\tbundle{\reals^{n}}{m}$; we use this metric throughout the rest of the proof.  For each $p\in \reals^{n}$, let $Q_{p}$ denote the orthogonal
projection operator on $\tbundle{\reals^{n}}{m}_{p}$ whose image is
$G_{p}$.

Let $p$ be a point of $\dimset{d}$.  We can find sections
$s_{1},\dots,s_{d}$ of $G$ defined on some open neighborhood of $p$ such
that $s_{1}(p),\dots,s_{d}(p)$ is a basis of $G_{p}$.  These sections
are linearly independent on some open neighborhood $U$ of $p$.
At each point $q\in U$, define
$D_{q}\subseteq G_{q}$ to be the span of $s_{1}(q),\dots,s_{d}(q)$.  For each $q\in U$,
let $P(q)$ be the orthogonal projection operator on
$\tbundle{\reals^{n}}{m}_q$ whose image is $D_{q}$.  At any point $q$
in $U\cap \dimset{d}$, we have $D_{q}=G_{q}$, and so $P(q)=Q_{q}$ at
such points.

We can think of $P$ as a vector bundle mapping of $\tbundle{U}{m}$ to
itself, or as a section of the bundle
$\Hom(\tbundle{U}{m},\tbundle{U}{m})$, whose fiber at $q$ is the vector
space $\Hom(\tbundle{U}{m}_{q},\tbundle{U}{m}_{q})$ of linear maps
$\tbundle{U}{m}_{q}\to \tbundle{U}{m}_{q}$.  Because
$\tbundle{U}{m}$ is a trivial bundle, $P$ can be thought of as a map $U\to
\Hom(\reals^{m},\reals^{m})$.  The mapping $q\mapsto P(q)$ is smooth.
Indeed, if we think of $P(q)$ as a linear map on $\reals^{m}$, its
matrix with respect to the standard basis can be explicitly
constructed by applying the Gram-Schmidt process to the vectors
$s_{1}(q),\dots, s_{d}(q)$, which shows that the matrix entries
are smooth functions of $q$.

The following lemma summarizes the discussion above.

\begin{lem}\label{thm:plemma}
For each $p\in \dimset{d}$, choose  a ball $B\in \ratballs$ such that $p\in B$ and $2B\subseteq U$.
There exists a smooth vector bundle map $P$ of
  $\tbundle{2B}{m}$ to itself with the following properties:
  For each $q\in 2B$,
  \begin{enumerate}
  \item $P(q)$ is an orthogonal projection operator.
    \item $\im(P(q))$, the image of $P(q)$, is contained in $G_{q}$.
      \item $\im(P(q))$ has dimension $d$.
      \item If $q\in 2B\cap \dimset{d}$ then $im(P(q))=G_{q}$
        and so $P(q)=Q_{q}$.
   \end{enumerate}
\end{lem}

Since $\ratballs$ is countable, the lemma shows that we can find a
countable collection of balls $\lset B\rset_{i\in I}$ that
covers $\dimset{d}$, where for each
ball $B_{i}$ there is a vector bundle map $P_{i}$ over $2B_{i}$ with
the properties in the lemma. There may be many such vector
bundle maps over a given ball $2B_{i}$; we
only need one, so we just pick one arbitrarily.

Let $e_{1},\dots,e_{m}$ denote the standard basis of $\reals^{m}$.  Let
$E_{1},\dots, E_{m}$ be the corresponding constant sections of
$\tbundle{\reals^{n}}{m}$.

For each $i\in I$, choose a smooth bump function $\psi_{i}$ on $\mathbb{R}^n$ so
that $0\leq \psi_{i}\leq 1$, $\psi_{i}=1$ on $B_{i}$ and
$\supp(\psi_{i})\subset 2B_{i}$.

On $2B_{i}$ we define smooth sections $P_{i}E_{\alpha}$ of $G$
by $p\mapsto
P_{i}(p)E_{\alpha}(p)$ for $\alpha=1,\dots,m$. 
Multiplying by $\psi_{i}$ we get sections
$\psi_{i}P_{i}E_{\alpha}$ such that
$\supp(\psi_{i}P_{i}E_{\alpha})\subset 2B_{i}$.  We extend the
section $\psi_{i}P_{i}E_{\alpha}$
smoothly to $\mathbb{R}^n$ by defining it to be zero
outside $2B_{i}$. We use the same
notation for the extended sections. We also adopt the notational
convention that
\begin{equation}
  \label{eq:3}
  0\cdot \text{(undefined)} =0
\end{equation}
in this context, as many authors do implicitly.

Let us deal first with the case where our collection of balls $\lset
B_{i}\rset_{i\in I}$ is countably infinite, in which case we can
assume the index set $I$ is the natural numbers.

Since $\psi_{i}$ has compact support, $\psi_{i}$ and its
derivatives are bounded, so $\psi_{i}\in \bcsp{\reals^{n}}{\reals}$.
Similarly, the sections $\psi_{i}P_{i}E_{\alpha}$ have compact
support, so we can view them as vector-valued functions
in $\bcspnm$.   

For each $i$, we can find a constant $c_{i}>0$ so that
\begin{align*}
  c_{i}\norm{\psi_{i}}_{i}&\leq \frac{1}{2^{i}},\\
  c_{i}\norm{\psi_{i}P_{i}E_{\alpha}}_{i}&\leq \frac{1}{2^{i}},\qquad
  \alpha=1,\dots, m.
\end{align*}
Note that the order of the seminorm is the same as the index
here.

Now define functions $\phi_{i}=c_{i}\psi_{i}$, so we can rewrite the
inequalities above as 
\begin{align*}
  \norm{\phi_{i}}_{i}&\leq \frac{1}{2^{i}},\\
  \norm{\phi_{i}P_{i}E_{\alpha}}_{i}&\leq \frac{1}{2^{i}},\qquad
  \alpha=1,\dots, m.
\end{align*}

We now attempt to define a smooth function $\phi$ and smooth sections
$S_{\alpha}$ of $G$ by
\begin{align*}
  \phi &= \sum_{i=1}^{\infty} \phi_{i}\, ,\\
  S_{\alpha} &= \sum_{i=1}^{\infty} \phi_{i} P_{i}E_{\alpha}.
\end{align*}
To do this, we must show these series are convergent in the
appropriate function spaces.

To show that the series $\sum_{i}\phi_{i}$ is convergent in
$\bcsp{\reals^{n}}{\reals}$, it suffices to show that the series
\begin{equation}\label{eq:1}
  \sum_{i=1}^{\infty} \norm{\phi_{i}}_{k}
\end{equation}
converges for each $k$.   To show that
the series \eqref{eq:1} converges, it suffices to show for fixed $k$ that
the tail
\begin{equation*}
  \sum_{i=k}^{\infty} \norm{\phi_{i}}_{k}
\end{equation*}
of the series converges.  Since the norms $\norm{\cdot}_{j}$ are
increasing in $j$, we have
\begin{equation*}
  \sum_{i=k}^{\infty} \norm{\phi_{i}}_{k} \leq \sum_{i=k}^{\infty}
\norm{\phi_{i}}_{i} \leq \sum_{i=k}^{\infty} \frac{1}{2^{i}}=2^{-k+1}.
\end{equation*}
We conclude that the series $\sum_{i}\phi_{i}$ converges in
$\bcsp{\reals^{n}}{\reals}$, so $\phi = \sum_{i}\phi_{i}$ is a smooth
function.  As previously mentioned, we can evaluate this series
pointwise, so at any point $p\in M$ we have
\begin{equation}\label{eq:2}
  \phi(p) = \sum_{i=1}^{\infty} \phi_{i}(p).
\end{equation}
It follows that $\phi > 0$ on $\dimset{d}$, since a point $p\in
\dimset{p}$ is in one of the balls $B_{j}$ and
$\phi_{j}=c_{j}\psi_{j}=c_{j}>0$ on $B_{j}$. Since all of the terms in
the sum \eqref{eq:2} are nonnegative, we conclude that $\phi(p)\geq
c_{j}>0$.

Similarly, we can show that the series $\sum_{i}\phi_{i}P_{i}E_{\alpha}$ is
convergent in $\bcspnm$.  As above, the sum
\begin{equation*}
  \sum_{i=1}^{\infty} \norm{\phi_{i}P_{i}E_{\alpha}}_{k}
\end{equation*}
converges since we have
\begin{equation*}
  \sum_{i=k}^{\infty} \norm{\phi_{i}P_{i}E_{\alpha}}_{k} \leq 
\sum_{i=k}^{\infty} \norm{\phi_{i}P_{i}E_{\alpha}}_{i}
\leq \sum_{i=k}^{\infty} \frac{1}{2^{i}} =2^{-k+1}.
\end{equation*}
Thus, we have smooth  vector-valued functions, or to look at it
another way, sections of the trivial bundle, defined by
\begin{equation*}
  S_{\alpha} = \sum_{i=1}^{\infty} \phi_{i}P_{i}E_{\alpha}.
\end{equation*}
We can evaluate this sum pointwise and write
\begin{equation*}
  S_{\alpha}(p) = \sum_{i=1}^{\infty} \phi_{i}(p) P_{i}(p) E_{\alpha}(p),
\end{equation*}
using the convention \eqref{eq:3}. Since the image of each $P_{i}$ is
in $G$, the sum on the right-hand side of the equation above is
a convergent series in the closed subspace $G_{p}\subseteq
\tbundle{\reals^{n}}{m}_{p}$, so $S_{\alpha}(p)\in G_{p}$. Thus,
$S_{\alpha}$ is a section of $G$.

We now claim that if $p\in \dimset{d}$, then the sections
$S_{\alpha}(p)$ span $G_{p}$.  Recall from Lemma~\ref{thm:plemma}
 that if $p\in 2B_{i}$ then
$\im(P_{i}(p))=G_{p}$ and $P_{i}(p)=Q_{p}$.
 We then have
\begin{align*}
  S_{\alpha}(p) & = \sum_{i=1}^{\infty} \phi_{i}(p) P_{i}(p)
  E_{\alpha}(p)\\
&=\sum_{i=1}^{\infty} \phi_{i}(p) Q_{p} E_{\alpha}(p)\\
&=\sum_{i=1}^{\infty} Q_{p}[\phi_{i}(p) E_{\alpha}(p)]
=Q_{p} [\phi(p) E_{\alpha}(p)].
\end{align*}
For each $p$, $\{E_{1}(p),\dots, E_{m}(p)\}$ is a basis of
$\tbundle{\reals^{n}}{m}_{p}$.  Since $\phi(p)\ne 0$,
\begin{equation*}
 \phi(p)
E_{1}(p),\dots,\phi(p) E_{m}(p)
\end{equation*}
also form a basis.  
Thus,
\[
 \{Q_{p}\phi(p)
E_{1}(p),\dots,Q_{p}\phi(p) E_{m}(p) \}
\]
spans $G_{p}$, and thus
the vectors $S_{\alpha}(p)$ span $G_{p}$.

This completes the construction of a finite number of generators
for $\left.G\right|_{\Sigma_d}$ in the case where we have
a countably infinite collection of balls.  In the case where our
collection of balls $\lset B_{i}\rset_{i\in I}$ is finite, we can
dispense with the convergence questions and just define
\begin{align*}
  \phi &= \sum_{i\in I}\psi_{i}\\
  S_{\alpha} &= \sum_{i\in I}\psi_{i} P_{i}E_{\alpha},
\end{align*}
where these are finite sums.  A similar analysis shows that the
sections $S_{\alpha}$ span $G_{p}$ at every $p\in \dimset{d}$.

Finally, to complete the proof of Proposition~\ref{thm:4}, we
apply this construction for each $d$ such that $\dimset{d}\ne
\emptyset$.  For each such integer $d$, we get $m$ sections.
Putting all these sections together we get a finite set of
globally defined sections that span $G$ at each point, i.e.,
a finite set of global generators for $G$. \qed\newline

To finish this section, we discuss the number of sections this
construction yields and the case of disconnected manifolds.

If we denote by $\maxdim(G)$ the maximum dimension of the fibers of $G$,
the above construction will yield $m$ sections for every integer
$d$, $1\leq d \leq \maxdim(G)$ such that $\dimset{d}\ne \emptyset$,
and so a maximum of $m \maxdim(G)$ sections.  If desired,
we can get exactly $m \maxdim(G)$ sections by adding in $m$ copies
of the zero section for each $d$ such that $\dimset{d}=\emptyset$.

Recall that our original vector bundle $E$ is
isomorphic to a subbundle of the trivial bundle $\tbundle{M}{m}$.  
As mentioned
above, an upper estimate on $m$ is $(\dim(M)+1)\fdim(E)$.
In the case where $E=TM$,
we can get a better estimate.  By the hard Whitney embedding
theorem, $M$ can be embedded in $\reals^{n}$ where $n=2\dim(M)$.
Then $TM$ is isomorphic to a subbundle of the restriction of
$T\reals^{n}$ to $M$.  Since $T\reals^{n}$ is canonically
trivial, we see that $TM$ is isomorphic to a subbundle of
a trivial bundle of fiber dimension $n$, so we can take $m=n=2\dim(M)$
in this case.

Different authors use slightly different definitions of
manifolds and vector bundles.   One can define a manifold so that
the dimension is allowed to be different at different points. In this
case the dimension is locally constant, and so must be constant on
components.  With this definition, a manifold $M$ that is not
connected can have components of different dimensions.  If the
number of components is infinite, it is conceivable the dimensions
of the components could be unbounded, although one might be hesitant
to use the word ``manifold'' in that case.

Similarly, one can give a definition of the concept of a vector bundle
$E$ over a manifold $M$ 
that allows the fiber dimension to vary with the point.  The local
triviality condition  makes the fiber dimension locally constant,
and so constant on the components of $M$.  This point of view
is taken in some of the foundational literature behind this paper,
such as Swan \cite{MR0143225}.  If $M$ has infinitely many components, $E$
could have fibers of arbitrarily large dimension.  

For each component $C$ of $M$, we can find an $m_{C}$ so that
$E_{C}=E\mid_{C}$ is isomorphic to a subbundle of
$\tbundle{C}{m_{C}}$.  If we have an upper bound on the dimension
of the components of $M$ and on the dimension of the fibers of $E$,
we can get an upper bound $m$ on $m_{C}$.  Then, for each $C$,
$E_{C}$ is isomorphic to a subbundle of $\tbundle{C}{m}$.
Since we have an upper bound of the dimension of the fibers of
$E$, there is an upper bound on the dimension of the fibers
of $G$.  As above, we can construct on each $C$ a finite
set of global generators $S^{C}_{j}$ for $j=1,\dots, m\,\maxdim(G)$.
We can then define global sections $S_{j}$, $j=1,\dots, m\,\maxdim(G)$
by defining $S_{j}(p)=S^{C}_{j}(p)$, where $C$ is the component
containing $p$.  Thus, we will still have a finite number of global
generators in this case.

\section{Modules of Sections}

Let $E$ be a vector bundle over $M$ of fiber dimension $k$.
  For any open set $U\subseteq M$,
the space $\sects{U}{E}$ of sections of $E$ over $U$ is a module over
the ring $\rfuns{U}$ of smooth functions on $U$.  Every point $p\in M$
has a neighborhood $U$ on which there are sections $s_{1},\dots,s_{k}$
such that $s_{1}(q),\dots,s_{k}(q)$ form a basis of $E_{q}$ for all
$q\in U$.  Thus, for an arbitrary section $s\in \sects{U}{E}$,
we have $s(q)=f_{1}(q) s_{1}(q)+\dots+f_{k}(q)s_{k}(q)$ for some
uniquely determined functions $f_{1},\dots, f_{k}$.  The definition of
vector bundle shows that these functions are smooth.  As elements
of the module $\sects{U}{E}$ we have $s=f_{1}s_{1}+\dots +f_{k}s_{k}$,
so $s_{1},\dots, s_{k}$ is  set of generators for $\sects{U}{E}$.

The space $\gsects{E}$ of global sections of $E$ is a module over the
ring $\rfuns{M}$ of smooth functions on $M$.  This module is also
finitely generated.  This fact is part of the proof that $E$ is
isomorphic to a subbundle of a trivial bundle; see Greub, Halperin
and Vanstone~\cite{MR0336650}*{p.~77} and Swan~\cite{MR0143225}.

One can ask the same questions about the modules of sections of a
generalized subbundle $G\subseteq E$.  The fact that we can find a
finite set of global generators for $G$ might lead one to hope that
the module $\gsects{G}$ is finitely generated.  
Alas, we will give an example to show that $\gsects{G}$ is
not in general finitely generated and that there may be arbitrarily
small neighborhoods $U$ of a point $p$ such that the module
$\sects{U}{G}$ is not finitely generated. 

To begin the construction of our example, we introduce some notation.
Let $J=(-a,a)\subseteq \reals$ be an open interval, where we allow the
case $a=\infty$.  Let $\idl{J}\subseteq\rfuns{J}$ be
the space of smooth functions on $J$ that are zero on $(-a,0]$.
Recall that one can construct a function $\psi\in \idl{\reals}$ by
\begin{equation*}
  \psi(x) =
  \begin{cases}
    e^{-1/x}, & 0<x<\infty\\
    0,  & -\infty < x \leq 0,
  \end{cases}
\end{equation*}
for example. The restriction of $\psi$ to $J=(-a,a)$ is an element of
$\idl{J}$.

\begin{prop}\label{thm:5}
Let $G$ be the generalized subbundle of
$\tbundle{\reals}{1}$ given by
\begin{equation*}
  G_{x} =
  \begin{cases}
    \tbundle{\reals}{1}_{x}, & x>0,\\
    \lset 0 \rset \subset \tbundle{\reals}{1}_{x}, & x \leq 0.
  \end{cases}
\end{equation*}
This is a smooth generalized subbundle of $\tbundle{\reals}{1}$; indeed
it is spanned by the single smooth section $\psi$.

The module of sections $\gsects{G}$ is not finitely generated, and
$\sects{J}{G}$ is not finitely generated for any interval $J=(-a,a)$, $a>0$.

\end{prop}

\begin{proof}
Regarding sections of the trivial bundle as functions, we have
$\sects{J}{G}=\idl{J}$. Clearly $\idl{J}$ is an ideal in
the ring $\rfuns{J}$, and our assertion is that $\idl{J}$ is not
finitely generated.
Suppose that $\mathcal{I}(J)$ is generated by the elements
$g_1,\ldots, g_k\in\mathcal{I}(J)$. Let
\[
r(x) =\left( \sum_{j=1}^k g_j^2(x)\right)^{1/2},
\]
which may be smooth or not. We have each $g_j(x)=\mathcal{O}(x^N)$ for every $N\in\mathbb{N}$ as $x\to 0$, so that $r(x)=\mathcal{O}(x^N)$ for every $N$ as $x\to 0$. 
By the assumption, for all $h\in \mathcal{I}(J)$, $h=\sum_j a_j g_j$ for some $a_j\in C^\infty(J)$. Applying the Cauchy-Schwarz inequality,
\[
|h(x)|\leq 
\left( \sum_{j=1}^k a_j^2(x)\right)^{1/2}r(x).
\]
Fix $M\in (0,a)$.
Then $a_j$ are bounded on $[0,M]$, so for
any $h\in\mathcal{I}(J)$, there exists a constant $C$ such that
\begin{equation}
|h(x)|\leq C|r(x)| \label{here}
\end{equation}
for all $x\in (0,M]$.

We will now obtain a contradiction by
constructing an element $h \in\mathcal{I}(J)$ that does not 
satisfy such an inequality.
Let $x_1>0$ be such that $x_1<\min\{1,\frac {2a}3\}$ and $r(x_1)<x_1^2$, and inductively define the sequence
$x_1>x_2>\ldots \to 0$ such that for all integers $n\geq 2$, $r(x_n)<x_n^{2n}$ and $x_{n+1}<
\frac {x_n}4$. (We are able to do this since $r(x)=\mathcal{O}(x^N)$ for all $N\in\mathbb{N}$.)

We now let $\phi$ be a $C^\infty$ bump function that is compactly supported in $(\frac 12,\frac 32)$ such that $\phi(x)\geq 0$, $\phi(1)=1$.
We define
\[
h(x) = \sum_{n=1}^\infty x_n^n \phi\left( \frac x{x_n}\right).
\]
Note that by construction, the supports of the terms are
subsets of
$\left( \frac {x_n}2,\frac{3x_n}2\right)$, which are pairwise disjoint.
Fix a positive integer $N$.
By construction, for any nonnegative integer $m$, there 
exists a constant $C_m$ such that if $x$ lies in the $n^{\mathrm{th}}$ 
interval (for $n\geq N+m$),
\[
|h^{(m)}(x)|\leq C_m x_n^{n-m}\leq C_{m,N}x^N,
\]
for a constant $C_{m,N}$. From this it follows that all derivatives 
of $h$ approach $0$ as $x\to 0^+$, so $h\in\mathcal{I}(J)$.

However, observe that $h(x_n)=x_n^n$ and $r(x_n)<x_n^{2n}$, so that
\[
\frac {h(x_n)}{r(x_n)}>x_n^{-n}\to\infty
\]
as $n\to\infty$. Thus, there does not exist $C>0$ such that
$|h(x)|\leq Cr(x)$, contradicting (\ref{here}). Therefore, $\mathcal{I}(J)$ is not finitely generated as a
$C^\infty(J)$-module.

\end{proof}

Proposition~\ref{thm:5} implies that $\rfuns{J}$ is not noetherian, which is no surprise.

Because $T\reals\cong \tbundle{\reals}{1}$, the proposition above
shows that tangent distributions are no better behaved
than generalized subbundles of arbitrary vector bundles in this
respect.


\end{document}